\documentclass[reqno, letterpaper, oneside]{article}

\usepackage{amssymb}
\usepackage{amsmath}
\usepackage{amsthm}

\usepackage{bbm}
\usepackage{setspace}
\usepackage{geometry}
\newcommand{\indic}[1]{\mathbbm{1}_{\{{#1}\}}}

\newtheorem{theorem}{Theorem}[section]
\newtheorem{lemma}[theorem]{Lemma}
\newtheorem{remark}[theorem]{Remark}
\newtheorem{conjecture}[theorem]{Conjecture}

\newtheorem{claim}[theorem]{Claim}

\newtheorem{question}[theorem]{Question}

\newcommand\mc{\mathcal}

   \def\D{\Delta}
 \def\f{\phi}

\def\cH{\mathcal{H}}
\def\cD{\mathcal{D}}
\def\cE{\mathcal{E}}
\def\cF{\mathcal{F}}

\def\cC{\mathcal{C}}

\def\cW{\mathcal{W}}
\def\cT{\mathcal{T}}
\def\cG{\mathcal{G}}

\def\cU{\mathcal{U}}

\def\fF{\mathfrak{F}}

\def\he{\hat{e}}
\def\hq{\hat{q}}
\def\hw{\hat{w}}
\def\hx{\hat{x}}
\def\hy{\hat{y}}

\def\tq{\tilde{q}}

\def\Pr{\mathbb{P}}

\def\E{\mathbb{E}}
\def\bF{\mathbb{F}}

\newcommand\mb{\mathbb}

\newcommand{\refT}[1]{Theorem~\ref{#1}}

\newcommand{\refL}[1]{Lemma~\ref{#1}}
\newcommand{\refR}[1]{Remark~\ref{#1}}
\newcommand{\refCl}[1]{Claim~\ref{#1}}
\newcommand{\refS}[1]{Section~\ref{#1}}

\newcommand{\refConj}[1]{Conjecture~\ref{#1}}
\newcommand{\refQ}[1]{Question~\ref{#1}}

\newcommand\set[1]{\ensuremath{\{#1\}}}
\newcommand\bigset[1]{\ensuremath{\bigl\{#1\bigr\}}}

\newcommand\biggset[1]{\ensuremath{\biggl\{#1\biggr\}}}

\newcommand\xpar[1]{(#1)}
\newcommand\bigpar[1]{\bigl(#1\bigr)}
\newcommand\Bigpar[1]{\Bigl(#1\Bigr)}
\newcommand\biggpar[1]{\biggl(#1\biggr)}
\newcommand\Biggpar[1]{\Biggl(#1\Biggr)}

\newcommand\bigsqpar[1]{\bigl[#1\bigr]}
\newcommand\Bigsqpar[1]{\Bigl[#1\Bigr]}
\newcommand\biggsqpar[1]{\biggl[#1\biggr]}
\newcommand\Biggsqpar[1]{\Biggl[#1\Biggr]}

\newcommand\bigcpar[1]{\bigl\{#1\bigr\}}
\newcommand\Bigcpar[1]{\Bigl\{#1\Bigr\}}
\newcommand\biggcpar[1]{\biggl\{#1\biggr\}}
\newcommand\Biggcpar[1]{\Biggl\{#1\Biggr\}}

\newcommand\bigabs[1]{\bigl|#1\bigr|}

\def\rompar(#1){\textup(#1\textup)}    

\def\xexp(#1){e^{#1}}
\newcommand\bigceil[1]{\big\lceil#1\big\rceil}
\newcommand\ceil[1]{\lceil#1\rceil}

\renewcommand{\emptyset}{\varnothing} 

\let\OLDthebibliography\thebibliography
\renewcommand\thebibliography[1]{
  \OLDthebibliography{#1}
  \setlength{\parskip}{0pt}
  \setlength{\itemsep}{0pt plus 0.3ex}
}

\begin{document}

\title{Large girth approximate Steiner triple systems} 
\author{Tom Bohman\thanks{Department of Mathematical Sciences, Carnegie Mellon
University, Pittsburgh, PA 15213, USA. Email: {\tt tbohman@math.cmu.edu}.} 
\and Lutz Warnke\thanks{School of Mathematics, Georgia Institute of Technology, 
Atlanta GA~30332, USA. E-mail: {\tt warnke@math.gatech.edu}.
Research partially supported by NSF Grant DMS-1703516 and a Sloan Research Fellowship.}}
\date{July 31, 2018; revised January 2, 2019} 

\renewcommand{\thefootnote}{\fnsymbol{footnote}}
\footnotetext{\hspace{-0.5em}AMS 2000 Mathematics Subject Classification: 05B07, 05C80, 60C05, 05B30, 60G99}
\renewcommand{\thefootnote}{\arabic{footnote}} 

\maketitle

\begin{abstract}
In~1973 Erd\H{o}s asked whether there are $n$-vertex partial Steiner triple systems with arbitrary high girth and quadratically many triples.   
(Here girth is defined as the smallest integer~$g \ge 4$ for which some~$g$-element vertex-set contains at least~$g-2$~triples.) 

We answer this question, by showing existence of approximate Steiner triple systems with arbitrary high girth. 
More concretely, for any fixed~$\ell \ge 4$ we show that 
a natural constrained random process typically produces a partial Steiner triple system with~$ (1/6 -o(1)) n^2 $ triples and girth larger than~$\ell$. 
The process iteratively adds random triples subject to the constraint that the girth remains larger than~$\ell$. 
Our result is best possible up to the~$o(1)$-term, which is a negative power of~$n$. 
\end{abstract}

\section{Introduction} 
A {\em Steiner triple system} is a $3$-uniform hypergraph~$\cH$ with the property that
every $2$-element subset of its vertex-set $V(\cH)$ is contained in exactly one triple of~$\cH$ 
(so~$\cH$ decomposes the complete graph with vertex-set~$V(\cH)$ into edge-disjoint triangles). 
Steiner triple systems and their many natural generalizations
are central to combinatorics, and have been studied since the work of Pl\"ucker, Kirkman,
and Steiner in the mid-nineteenth century (see \cite{history} for discussion of this history, and~\cite{designs1,designs3,designs4,designs2} for recent breakthroughs).

In this paper we consider a `high-girth' generalization of Steiner triple systems proposed by Erd\H{o}s~\cite{conjecture}.
We define the \emph{girth} of a $3$-uniform hypergraph to be the smallest~$g \ge 4$ 
for which there is a set of~$g$ vertices that spans at least~$g-2$ triples. 
All known constructions of Steiner triple systems have small girth (see, e.g.,~\cite{AntiPasch,Sparse5,Sparse5b,Sparse6}), 
and there seems to be no simple reason to believe that this should be necessary. 
It thus is natural to ask whether or not there are Steiner triple systems of arbitrarily large girth (also called `locally sparse'). 
Simple divisibility reasons enforce the (well-known) necessary condition $n \equiv 1,3 \mod 6$ below.  
\begin{question}[Erd\H{o}s, 1973]
\label{quest:proper}
Let $ \ell \ge 4 $.  Does there exist $n_0 = n_0(\ell) $ such that there are 
$n$-vertex Steiner triple systems with girth greater than~$\ell$ for every~$n \ge n_0$ with~$n \equiv 1,3 \mod 6$?%
\end{question}
\noindent
This question remains largely open; see~\cite{AntiPasch,Sparse5,Sparse5b,Sparse6} for the most recent
developments (which answer~this question only for~$\ell \le 6$). 
Erd\H{o}s~\cite{conjecture} also asked an approximate version of \refQ{quest:proper}.  
A~$3$-uniform hypergraph~$\cH$ is called a {\em partial Steiner triple system} 
if every $2$-element vertex-subset is contained in at most one triple of~$\cH$.  
Note that any partial Steiner triples
system on $n$ vertices has at most~$\frac{1}{3} \binom{n}{2} = n^2/6 - \Theta(n)$~triples.
\begin{question}[Erd\H{o}s, 1973]
\label{quest:approx}
For which~$\ell \ge 4$ and~$c \in (0,1/6)$ are there $n$-vertex partial Steiner triple systems  
with at least~$c n^2$ triples and girth larger than~$\ell$ for all $n \ge n_0(\ell,c)$? 
\end{question}
\noindent
In~1993 Lefmann, Phelps, and R{\"o}dl~\cite{LPR} showed that for any~$\ell \ge 4$ one can take~$c=c_\ell>0$ with~$c_\ell \to 0$ as~$\ell \to \infty$, 
and raised the question whether one can take a constant~$c>0$ that does not depend on~$\ell$. 
This natural question was also formulated more recently by Ellis and Linial~\cite{EL} (see also~\cite{S}).

In this paper we answer Erd\H{o}s' \refQ{quest:approx}, 
by showing existence of approximate Steiner triple systems with arbitrary high girth. 
Regarding the above-mentioned 
questions from~\cite{LPR,EL}, this implies that one can take~$c_\ell \sim 1/6$ for any~$\ell \ge 4$. 
These results 
 were obtained independently by Glock, K\"uhn, Lo, and Osthus~\cite{GKLO}.\footnote{The quantitative result of this paper is slightly stronger than~\cite{GKLO}, which gives~$(1-o(1))n^2/6$ instead of~$(1-n^{-\beta_\ell})n^2/6$~triples.} 
\begin{theorem}[Main result]%
\label{thm:top}%
For every~$\ell \ge 4$ there are~$n_\ell,\beta_\ell>0$ such that, for all $n \ge n_\ell$, 
there exists an $n$-vertex partial Steiner triple system
with at least~$\bigpar{1-n^{-\beta_\ell}} n^2/6$ triples 
and girth larger than~$\ell$.  
\end{theorem}
\noindent
We prove Theorem~\ref{thm:top} by showing that the following natural (see~\cite{KKLS}) constrained random process is very likely to produce the
desired object for fixed~$ \ell \ge 4$ (see \refT{thm:main}). 
Beginning with the empty $3$-uniform hypergraph~$\cH_0$ on $n$~vertices 
we sequentially set~$ \cH_{i+1} := \cH_i + e_{i+1}$, where the added triple~$e_{i+1}$ is chosen uniformly 
at random from the collection of triples~$xyz\not\in \cH_i$ with the property that the girth of~$\cH_i + xyz$ remains larger than~$\ell$ 
(i.e., that~$\cH_i + xyz$ contains no set of~$4 \le a \le \ell$ vertices that spans at least~$a-2$ triples). 
This process terminates with a maximal partial Steiner triple system with girth larger than~$\ell$.

Our differential equation method~\cite{nick2,r3t,BK,Wa3,FGM} based analysis of this random process 
is motivated by a pseudo-random heuristic for divining
the trajectories that govern the evolution of various key parameters (see~\refS{sec:pseudo-random}).  
Such heuristics play a central role
in our understanding of several other constrained random processes that produce interesting
combinatorial objects (such as the triangle-free process~\cite{r3t,Wz3,BK,FGM,GW}, 
the triangle removal process~\cite{BFL0,BFL}, 
and the $H$-free process~\cite{BK2,Wa1,Wa3,Wa2,P3}). 
A surprising consequence of our proof is that the general case only introduces minor modifications 
of the trajectories compared to the $\ell=4$ case (see~\refR{rem:main}), i.e., 
adding the arbitrary high-girth constraint does not affect the evolution~significantly.

Let us briefly mention that the notion of girth introduced above is the sparsest for which Questions~\ref{quest:proper}--\ref{quest:approx} are feasible.
Firstly, is is easy to check that any $n$-vertex Steiner triple system contains, for all~$3 \le g \le n$, a set of~$g$ vertices that spans at least~$g-3$ triples. 
Secondly, a celebrated result of Rusza and Szemer{\'e}di~\cite{RZ} states that any $n$-vertex $3$-uniform hypergraph in which no~$6$~vertices span~$3$~or~more triples contains at most~$o(n^2)$ triples. 
Thirdly, 
Lefmann, Phelps, and R{\"o}dl \cite{LPR} showed that any $n$-vertex Steiner triple system contains, for some~$g = O(\log n/\log \log n)$, a set of~$g$ vertices that spans at least~$g-2$~triples.

The remainder of this paper is organized as follows.  In the next section we
formally introduce the high-girth triple-process, and use a pseudo-random heuristic
to predict evolution trajectories for the key parameters of the process.  This
leads to the statement of our main technical result, \refT{thm:main}, which
we then prove in \refS{sec:proof}. 
The final \refS{sec:conclusion} contains some brief concluding remarks and conjectures.

\section{The high-girth triple-process}\label{sec:high-girth}
Henceforth fixing $\ell \ge 4$, our notational conventions are as follows.  
Let $\fF^+=\fF^+_{\ell}$ denote the collection of all $3$-uniform hypergraphs~$F$ with~$4 \le v_F \le \ell$ vertices and~$e_F=v_F-2$ triples, 
with the property that~$F$ contains no subhypergraph $J \subsetneq F$ with~$v_J \ge 4$ and~$e_J=v_J-2$. 
Let $\fF=\fF_{\ell}$ be the restriction of~$\fF^+$ to graphs~$F$ with~$v_F \ge 6$ vertices. 
\begin{remark}\label{rem:cF}
It is an easy exercise to check that any $3$-uniform hypergaph~$F$ with~$v_F =5$ vertices and $e_F=v_F-2$ triples
contains a subhypergraph $J \subseteq F$ with~$v_J=4$ vertices and~$e_J =2$ triples. 
\end{remark}
\noindent 
The \emph{high-girth triple-process} can now be defined as follows.  We begin with $ \cH_0$ which is the empty
$3$-uniform  hypergraph on $n$ vertices.  The triple $xyz$ is {\em available} at step~$i$ if~$ xyz $ intersects every
previously chosen triple in at most one vertex and~$\cH_i + xyz$ does not contain any hypergraph in the 
collection~$\fF$.  The triple~$e_{i+1}$ is chosen uniformly at random from the collection of triples available
at step~$i$, and it is added to the hypergraph to form $\cH_{i+1} := \cH_i + e_{i+1}$.  
The process terminates at a hypergraph~$\cH_m$ that has no available triples.

The reader will note that it would be equivalent to declare a triple $xyz\not\in \cH_i$ available if $\cH_i + xyz$ does not contain any hypergraph in the 
collection~$\fF^+$.  We treat the obstruction given by two triples that intersect in two
vertices differently because of the important role it plays in the process.  Indeed, we will see that 
this 4-vertex `diamond' obstruction is the main driver of the evolution of the process (see \refR{rem:main}).

\enlargethispage{\baselineskip} 

For~$\ell=4$ this forbidden diamond subhypergraph 
is the only obstruction, and the resulting process is equivalent to the so-called random \emph{triangle removal process}, 
which generates a sequence $ \binom{[n]}{2} = E_0 \supset E_1 \supset \dots $ of subgraphs of the complete $n$-vertex graph. 
Given~$E_i$, this 
process chooses~$xyz$ uniformly at random from the collection of all
triangles in the graph~$G_i:=([n], E_i)$, and the triangle~$xyz$ is removed from the edge-set to form~$E_{i+1}:=E_i \setminus \{xy,xz,yz\}$.  Note that the $3$-uniform 
hypergraph $ \cH_i$ given by the set of triples chosen in the first~$i$~steps of random triangle removal is equal to the hypergraph
produced by the high-girth triple-process with~$\ell=4$.  
We shall use this analogy in our study of the high-girth triple-process with arbitrary~$\ell \ge 4$.

\subsection{Key variables}\label{sec:variables}
Our main goal is to understand the evolution of the number of available triples, i.e., the size of  
\begin{equation}\label{def:Qi}
Q(i) :=  \biggset{ xyz \in \binom{[n]}{3}: \ \text{$xyz$ is available at step~$i$}} .
\end{equation}
Note that at the beginning of the process~$Q(0) = \binom{[n]}{3}$ is complete, and that at termination of the process~$Q(m)$ is empty.  
In order to track~$|Q(i)|$ we need to handle the step-wise impact of the various structural obstructions. 
To this end it will be convenient to study the associated graph~$G(i)$ with vertex-set~$[n]$ and edge-set
\begin{equation}\label{def:Ei}
E(i) := \binom{[n]}{2} \setminus \bigcup_{ xyz \in \cH_i} \{ xy,xz, yz \} . 
\end{equation}
In words, $G(i)$ is the graph given by the pairs that do not appear in any triple of~$\cH_i$. 
Note that any triple~$xyz \in Q(i)$ satisfies~$\set{xy,zy,xz} \subseteq E(i)$ (but the converse of this statement does not hold for~$\ell \ge 5$). 
Hence, when~$xyz \in Q(i)$ is added to~$\cH_i$, triples~$abc \in Q(i)$ can become unavailable in~$\cH_{i+1}= \cH_i + xzy$ because some edges in~$\set{ab,ac,bc} \cap \set{xy,xz, yz} \subseteq E(i)$ are no longer in~$E(i+1)$.  
To account for such `removed' triples, for every edge~$uv \in E(i)$ of~$G(i)$ we thus introduce 
\begin{equation*}
Y_{uv}(i) :=  \bigset{ z \in [n] \setminus \set{u,v}: \ uvz \in Q(i) } . 
\end{equation*}
In words, $|Y_{uv}(i)|$ is the `available' codegree of~$uv \in E(i)$. 
Note further that when~$xyz \in Q(i)$ is added to~$\cH_i$, triples~$abc \in Q(i)$ can also become unavailable in~$\cH_{i+1}= \cH_i + xzy$ because~$\cH_{i}+\set{xyz,abc}$ contains a copy of some forbidden hypergraph~$F \in \fF$.  
To handle such `closed' triples we need to account for all possible `routes' to copies of~$F \in \fF$ in~$(\cH_i)_{i \ge 0}$: 
for all $F \in \fF$, all triples~$uvw \in Q(i)$ and~$0 \le k \le e_F-2$ we thus introduce 
\begin{align*}
\bF_F &:= \biggset{F' \subseteq \binom{[n]}{3} : \ F' \cong F}, \\
W_{uvw, F, k}(i) &:=  \bigset{F' \in \bF_{F}: \  uvw \in F' \cap Q(i), \ |F' \cap Q(i)|=e_F-k, \text{ and } |F' \cap \cH_i|=k}. 
\end{align*}
In words, $\bF_F$ is the set of all copies of~$F$ in the complete $n$-vertex $3$-uniform hypergraph, 
and~$W_{uvw, F, k}(i)$ is the set of all `extensions' of~$uvw \in Q(i)$ to copies of~$F$ which have~$e_F-k$ available triples and~$k$ triples in~$\cH_i$.

\subsection{Pseudo-random intuition: Trajectory equations}\label{sec:pseudo-random}
In this subsection we introduce our pseudo-random intuition and use it
to derive the trajectories $|Q(i)| \approx \hq(t)$, $|Y_{uv}(i)| \approx \hy(t)$, and $|W_{uvw, F, k}(i)| \approx \hw_{F, k}(t)$ 
that we expect the random variables to follow.  The trajectories have a continuous time variable $t$ that we relate to the
discrete steps in the process by setting
\begin{equation}
t = t(i,n) := i/n^2. 
\end{equation}
Note that we would like to follow the evolution of the process until~$t = 1/6 -o(1)$; in particular, we would
like to show that the process does not terminate before that point in time. 
In view of~$|\cH_i|=i \approx \binom{n}{3} \cdot 6t/n$, our 
\emph{pseudo-random ansatz} is that the triples of~$\cH_i$ are approximately independent with
\[ \Pr(uvw \in \cH_i) \approx 6t/n =: \pi(t,n) = \pi, \]
where independence only holds with respect to statistics of~$\cH_i$ that do not involve obstructions from~$\fF^+$.

As a first application of this heuristic, we consider the probability that a pair~$xy$ is in the
edge-set~$E(i)$ defined in~\eqref{def:Ei}. 
 Note that any two triples in the set $\set{ xyz: z \in [n] \setminus \set{x,y}}$ would form a `diamond' 
obstruction given by two triples that intersect in two vertices.  
Hence the events that these triples appear in~$ \cH_i$ are
pairwise disjoint (by construction of the high-girth process). 
In view of~\eqref{def:Ei}, our pseudo-random ansatz thus suggests 
\[ \Pr\bigpar{ xy \in E(i)} = \Pr\Bigpar{\bigcap_{z \in [n] \setminus \set{x,y}} \bigset{xyz \not\in \cH_i}} = 1 - \sum_{z \in [n] \setminus \set{x,y}} \Pr(xyz \in \cH_i)
\approx 1 - (n-2) \pi \approx 1 -6t.\]
Analogous to the triangle removal process analysis~\cite{BFL0,BFL} we therefore define
\begin{equation}
 p = p(t) := 1-6t ,
\end{equation}
and expect the edges of $G(i)=([n], E(i))$ to appear approximately independently with probability~$p$.

We now derive the trajectory for~$|Q(i)|$. 
Since~$uvw \in Q(i)$ if and only if~$\set{uv,uw,vw} \subseteq E(i)$ and~$\cH_i+uvw$ contains no~$F' \in \bigcup_{F \in \fF} \bF_F$ containing~$uvw$, 
our pseudo-random ansatz loosely suggests 
\begin{equation*}
\E |Q(i)| 
\approx \sum_{uvw \in \binom{[n]}{3}} \biggsqpar{p^3 \cdot \Pr\Bigpar{\bigcap_{F \in \fF} \bigset{\text{$\cH_i+uvw$ contains no~$F' \in \bF_F$ containing~$uvw$}}}} .
\end{equation*}
Note that~$|\bF_F| = n (n-1) \cdots (n-v_F+1)/|{\rm Aut}(F)| = n^{v_F}/|{\rm Aut}(F)| + O(n^{v_F-1})$, where ${\rm Aut}(F) $ is the automorphism group of~$F$. 
Since~$v_F-3=e_F-1$ for~$F \in \fF$, it follows (using symmetry and a double-counting argument) that the total number of~$F' \in \bF_F$  
containing some fixed triple~$uvw$ is 
\begin{equation}\label{eq:NuvwF} 
N_{uvw,F} := |\set{ F' \in \bF_F: \ uvw \in F'}| = \frac{|\bF_F| e_F}{\binom{n}{3}} = 
\frac{ 6 e_F}{ | {\rm Aut}(F) | }n^{e_F-1}   + O(n^{e_F-2}). 
\end{equation}
Note that if~$ uvw \not\in \cH_i$ and $ F_1, \dots, F_j \in \bigcup_{F \in \fF} \bF_F$ are hypergraphs that all contain the triple~$uvw$ and satisfy~$ V(F_r) \cap V(F_s) = uvw$ for all~$ r \neq s$, 
then~$\bigcup_{1 \le r \le j} F_r \setminus \{uvw\}$ contains no obstruction in~$\fF^+$ (exploiting minimality of these obstructions). 
It follows that we can
combine our pseudo-random ansatz with the well-known Poisson paradigm (which suggests that the appearance of the different obstructions~$F' \in \bigcup_{F \in \fF} \bF_F$ are approximately independent).
We thus anticipate 
\begin{equation*}
\E |Q(i)|  \approx \binom{n}{3} \cdot p^3 \cdot \prod_{F \in \fF}   \bigpar{1 -  \pi^{e_F-1}}^{ \frac{6 e_F }{ | {\rm Aut}(F) | }n^{e_F-1} } \approx \binom{n}{3} \cdot p^3 \cdot q 
\end{equation*}
and arrive at the idealized trajectory~$|Q(i)| \approx  \binom{n}{3} \cdot p^3 \cdot q \approx p^3 q n^3/6$, 
where 
\begin{equation}\label{def:q}
q(t) := \exp\biggcpar{-\sum_{ F \in \fF} \frac{6 e_F }{ | {\rm Aut}(F) | } (6t)^{e_F-1} } .
\end{equation}
\begin{remark}\label{rem:FL}
It is easy to see that~$|\fF^+_\ell| \le \sum_{4 \le r \le \ell}\binom{\binom{r}{3}}{r-2} \le \ell^{2\ell}$, and~$q(t) \ge \exp(-\ell^{2\ell})$ for~$t \in [0,1/6]$. 
\end{remark}
\begin{remark}\label{rem:main}
It follows from Remark~\ref{rem:FL} that the function $q(t)$ does not vanish as $t$ approaches $1/6$ for any fixed value of $\ell$.  Thus, if our
pseudo-random intuition is correct, the obstructions on more than 4 vertices introduce a negligible
alteration of the trajectory that governs the evolution of random triangle removal.  Indeed, as~$ t \to 1/6$
the random triangle removal edge probability~$p(t) $ goes to zero while the triple availability probability~$ q(t)$ does not. 
Therefore 
the larger 
obstructions do not cause the process to
terminate before almost all pairs are covered by the triples selected by the process.  This is a key
insight in this work. 
\end{remark}

Next, by similar pseudo-random considerations as for~$|Q(i)| \approx \binom{n}{3} \cdot p^3 \cdot q$ above, 
we also expect that the available codegree of~$uv \in E(i)$ satisfies~$|Y_{uv}(i)| \approx (n-2) \cdot p^2 \cdot q$,   
and that the extension variables of~$uvw \in Q(i)$ satisfy~$|W_{uvw, F, k}(i)| \approx  N_{uvw,F} \binom{e_F-1}{k} \cdot \pi^{k} \cdot (p^3q)^{e_F-1-k}$ for~$0 \le k \le e_F-2$.  
In view of~\eqref{eq:NuvwF} and~$\pi = 6t/n$, it follows that the idealized trajectories of~$|Q|$, $|Y_{uv}|$, and~$|W_{uvw, F, k}|$ ought to be 
\begin{align}
\label{def:hq}
\hq = \hq(t) & :=  p^3 q n^3/6,  \\
\label{def:hy}
\hy = \hy(t)  & :=   p^2 qn, \\
\label{def:hw}
\hw_{F, k} = \hw_{F, k}(t) & : =    \frac{ 6 e_F}{ | {\rm Aut}(F) | }  \binom{e_F-1}{k} (6t)^k (p^3qn)^{e_F-1-k} ,
\end{align}
where here and elsewhere we often suppress the dependence on~$t$ or~$n$ in the notation (to avoid clutter).


\subsection{Main technical result: Dynamic concentration}
In this subsection we state our main technical result for the high-girth triple-process, which implies~\refT{thm:top} and verifies our pseudo-random intuition. 
In particular, it shows that the random variables indeed closely follow the heuristic trajectories 
$|Q(i)| \approx \hq(t)$, $|Y_{uv}(i)| \approx \hy(t)$, and $|W_{uvw, F, k}(i)| \approx \hat{w}_{F, k}(t)$ from~\eqref{def:hq}--\eqref{def:hw}. 
\begin{theorem}[Main technical result]\label{thm:main}
For every $\ell \ge 4$ and $\tau > 0$ there exist constants~$\alpha,\beta \in (0,1)$ and $A,n_0>0$ (where $\alpha,\beta,A$ depend only on~$\ell$) such that  
for $n \ge n_0$,  with probability at least~$1 - n^{-\tau} > 0$, we have $|Q(i)|>0$ and 
\begin{align}
\label{eq:Qi}
|Q(i)| & = \hq(t) \pm  p^{-A}n^{\alpha} (p n^2),  & & \\
\label{eq:Yuvi}
|Y_{uv}(i)| &= \hy(t) \pm  p^{-A}n^{\alpha}    &&\text{for all $uv \in E(i)$}, \\
\label{eq:WuvwFk}
|W_{uvw,F,k}(i)| &= \hw_{F,k}(t) \pm p^{-A}n^{\alpha} ( p^2 n)^{e_F-(2+k)}  & & \text{for all $uvw \in Q(i)$, $F \in \fF$, $0 \le k \le e_F-2$}, 
\end{align}%
for all $0 \le i \le m_0:= \bigceil{\bigpar{1-n^{-\beta}} n^2/6}$. 
\end{theorem}
\begin{remark}
The proof shows that~\eqref{eq:Qi}--\eqref{eq:Yuvi} imply~$|Q(i)| \sim \hq(t)$ and~$|Y_{uv}(i)| \sim \hy(t)$ for~$0 \le i \le m_0$;~see~\eqref{eq:heQY}.  
\end{remark}
\noindent
We reiterate that the key observation is that the availability function~$q(t)$ 
does not tend to~$0$ as the high-girth process comes to an end near~$t =1/6$ (see \refR{rem:main}). 
This shows that the influence of the obstructions~$F \in \fF$ 
does not play a significant role in the evolution of the process 
(as they only alter the various trajectories by constant factors), 
confirming the Poisson paradigm developed in Section~\ref{sec:pseudo-random}; see also \refConj{conj:final}.  

\section{Analyzing the high-girth process: Proof of \refT{thm:main}}\label{sec:proof}
In this section we prove \refT{thm:main} by showing~$\Pr(\neg \cG_{m_0}) = o(n^{-\tau})$, 
where~$\cG_j$ is the event that~$|Q(i)|\ge \hq(t)/2$ and the estimates~\eqref{eq:Qi}--\eqref{eq:WuvwFk} hold for all $0 \le i \le j$. 
Deferring the choice of~$A \ge 2\ell$, for concreteness we set 
\begin{equation}\label{def:alpha:beta}
\alpha := \frac{3}{4} \quad \text{ and } \quad \beta := \min\biggcpar{\frac{1-\alpha}{2(A+2)}, \; \frac{\alpha}{80\ell}} . 
\end{equation}
For brevity we also introduce the error functions 
\begin{align*}
e_Y(t) := p^{-A}n^{\alpha}, \qquad e_Q(t):= e_Y \cdot (p n^2), \quad \text{ and } \quad e_{W_{F,k}} &:= e_Y \cdot ( p^2n )^{e_F-(2+k)}, 
\end{align*} 
so that we can write each of~\eqref{eq:Qi}--\eqref{eq:WuvwFk} in the form~$|X(i)|  = \hat{x}(t) \pm e_X(t)$. 
Furthermore, we shall 
always tacitly assume~$0 \le i \le m_0$, and that~$n \ge n_0(\ell,\tau,\alpha,\beta,A)$ is sufficiently large (whenever necessary). 
In particular, 
\begin{equation}
\label{eq:heQY}
\he_Y(t) :=e_Y(t)/\hy(t) = o(1)
 \quad \text{ and } \quad 
\he_Q(t) :=e_Q(t)/\hq(t) = o(1) ,
\end{equation}
since both functions are  of order $1/(p^{A+2}qn^{1-\alpha}) = O(n^{\beta(A+2)-(1-\alpha)}) = o(1)$ by choice of~$\beta$. 
Here and elsewhere we use the convention that all implicit constants may depend on~$\ell$ (but not on $\tau,\alpha,\beta,A,t$).

We first show that the bounds  for $|Q(i)|$ follow from the bound~\eqref{eq:Yuvi} for~$|Y_{uv}(i)|$. 
Indeed, using a double-counting argument (exploiting that~$xyz \in Q(i)$ satisfies~$\set{xy,zy,xz} \subseteq E(i)$) 
together with $|E(i)|=\binom{n}{2}-3i = [1-1/(np)] \cdot n^2p/2$ and~\eqref{eq:heQY}, 
it follows that $(\hy \pm e_Y)/(np) = O(\hy/(np)) = O(1) = o(e_Y)$ and 
\[
|Q(i)| = \frac{1}{3} \sum_{e \in E(i)} |Y_e(i)| = \frac{1}{3} \cdot \biggsqpar{\binom{n}{2}-3i} \cdot \bigpar{\hy \pm e_Y} \in \hq(t) \pm  e_Q .
\]
Using again~\eqref{eq:heQY}, now~$|Q(i)|\ge \hq(t)/2$ follows, with room to spare.

In the remainder of this section it thus suffices to establish the bounds~\eqref{eq:Yuvi}--\eqref{eq:WuvwFk} 
for the available codegree~$|Y_{uv}(i)|$ and the extension variables~$|W_{uvw,F,k}(i)|$. 
To this end, following the differential equation method approach to dynamic concentration, 
for each variable $X$ of the form~$Y_{uv}$ or~$W_{uvw,F,k}$ we introduce a pair of sequences of (auxiliary) random variables 
\begin{equation*}
X^\pm(i) := \pm\bigsqpar{|X(i)| - \hx(t)} - e_X(t) .
\end{equation*}
Note that the desired estimate~$|X(i)| = \hat{x}(t) \pm e_X(t)$ follows  
if the two estimates~$X^{\pm}(i) \le 0$ both hold. 
To show that $X^{\pm}(i) \le 0$ holds with sufficiently high probability, 
in Sections~\ref{sec:ECh}--\ref{sec:trend} we first establish that the sequences~$X^\pm(i)$ are supermartingales, 
and then provide bounds on the one-step changes $\Delta X^\pm(i) := X^{\pm}(i+1)-X^{\pm}(i)$ in \refS{sec:bounded} (see also \refS{sec:crude}). 
After noting that initially $X^{\pm}(0) \le -e_X(0)/2 = - n^{\Omega(1)}$ holds, 
in \refS{sec:martingale} we then use an Azuma--Hoeffding type inequality 
to show that~$X^{\pm}(i) \ge 0$ has extremely low probability 
(even in comparison with the polynomial number of such bad events).

\subsection{Expected one-step changes}\label{sec:ECh}
We begin by deriving expressions for the one-step expected changes 
of~$|Y_{uv}|$ and~$|W_{uvw,F,k}|$. 
For technical reasons we `freeze' these variables as soon as the 
relevant structural constraints 
from~\eqref{eq:Yuvi}--\eqref{eq:WuvwFk} are violated, i.e., 
we formally set~$|Y_{uv}(i+1)|=|Y_{uv}(i)|$ if $uv \not\in E(i+1)$, and~$|W_{uvw,F,k}(i+1)|=|W_{uvw,F,k}(i)|$ if $uvw \not\in Q(i+1)$. 
For brevity we write~$\D X(i) := |X(i+1)| - |X(i)|$, 
and henceforth tacitly assume~$0 \le i < m_0$.

For the changes $\D Y_{uv}(i) := |Y_{uv}(i+1)| - |Y_{uv}(i)|$ of the available codegree we assume that~$uv \in E(i)$ holds.  
In order to calculate the one-step expected change in this variable, we consider $ z \in Y_{uv}(i) $ and the event that $ uvz$ becomes unavailable in the
next step of the process. 
This occurs (recalling the discussion from \refS{sec:variables}) if the process chooses any triple associated with any of the following sets:
\[ Y_{uv}, Y_{uz}, Y_{vz} , \text{ or }  \bigcup_{F \in \fF} W_{uvz, F, e_F-2} . \] 
There are a couple of caveats that we need to bare in mind.  
If the process chooses some triple~$uvz'$ then we freeze the variable~$|Y_{uv}|$, and so we should not consider the influence of such triples. 
Furthermore, we would like to estimate the cardinality of
$ \bigcup_{F \in \fF} W_{uvz, F, e_F-2}  $ with the sum $  \sum_{ F \in \fF} \left|  W_{uvz, F, e_F-2} \right| $. 
However, this might not be correct since a triple~$h \neq uvz$ that is associated with the union could be
counted more than once by the sum, since~$\cH_i + \set{uvz,h}$ could contain multiple hypergraphs~$F$ in~$\fF$ 
(which includes the possibility of multiple copies of one of these hypergraphs).  
In order to deal with this possibility we introduce the {\em destruction fidelity} term 
\begin{equation}\label{eq:upsilon}
\Upsilon_{uvz}(i) := \bigabs{\bigcpar{h \in Q(i) \setminus \set{uvz} : \: \cH_i + \set{uvz,h} \text{ contains multiple }  K \in \fF \text{ containing~$\set{uvz,h}$}} }. 
\end{equation} 
So, noting that for distinct $e,e' \in \{uv,uz,vz\}$ there is at most one triple that is associated with both~$Y_e$ and~$Y_{e'}$, 
and no triples which are associated with both~$Y_e$ and~$W_{uvz, F, e_F-2}$ (as otherwise~$F \in \fF$ would contain a forbidden `diamond' subhypergraph~$\cD \subseteq F$ with~$v_\cD=4$ and~$e_\cD=2$), 
it follows that we can write 
\begin{equation}\label{eq:DYuv}
\E(\D Y_{uv}(i) \mid \cF_{i}) =  - \frac{1}{|Q|}  \sum_{z \in Y_{uv}} \Biggsqpar{ |Y_{uz}| + |Y_{vz}| + O(1) + \sum_{F \in \fF} |W_{uvz,F, e_F-2}| + O(\Upsilon_{uvz})},  
\end{equation} 
where~$\cF_i$ denotes, as usual, the natural filtration associated to our random hypergraph process after~$i$~steps. 
(We discuss a bound on $\Upsilon_{uvz}$ after considering the changes in the~$|W_{uvw,F,k}|$, see \refCl{cl:fidelity} below.)

For the one-step changes~$\D W_{uvw,F,k}(i) := |W_{uvw,F,k}(i+1)| - |W_{uvw,F,k}(i)|$ of the extension variables we assume that~$uvw \in Q(i)$ holds. 
There are two ways in which~$W_{uvw,F,k}$ changes: an extension~$F'$ currently in this collection can leave the collection if one of the available triples in~$F' \setminus \set{uvw}$ is chosen or becomes unavailable, 
and an extension~$F' \in W_{uvw,F,k-1} $ can move into~$W_{uvw,F,k}$ if one of the available triples in~$F' \setminus \set{uvw}$ is 
chosen by the process. 
Here a caveat is that we freeze the variable~$|W_{uvw,F,k}|$ when the process selects either~$uvw$ or a triple which makes~$uvw$ unavailable  
(which we anticipate to have negligible impact). 
Taking into account all these effects, it follows that 
\begin{equation}\label{eq:DWuvwFk}
\begin{split}
\E(\Delta W_{uvw,F,k}(i) \mid \cF_{i})  & = -   \frac{1}{|Q|}  \sum_{F' \in W_{uvw,F,k} } \Biggsqpar{  \sum_{\substack{f \in F' \cap Q:\\ f \neq uvw}}  \Biggpar{ \sum_{ xy \in \binom{f}{2}}  |Y_{xy}| +
O(1) + \sum_{ K \in \fF}  |W_{f,K,e_{K}-2}|  +O(\Upsilon_f)  }  +O(\Psi_{F'}) } \\
& \qquad 
+ \frac{\indic{k \ge 1}}{|Q|} \Biggsqpar{\sum_{F' \in W_{uvw,F,k-1} } \bigpar{e_F-k}  + O(\Lambda_{uvw,F,k-1})} ,
\end{split}
\end{equation}
where the three fidelity terms correspond to the following possibilities for overcounting: 
\begin{description}
\itemsep 0.125em \parskip 0em  \partopsep=0pt \parsep 0em 
\item[$\Upsilon_f$:] 
This term is identical to the destruction fidelity term introduced in~\eqref{eq:upsilon} above: 
it intuitively accounts for triples~$h \neq f$ whose selection 
would make the triple~$f \in Q(i)$ unavailable in more than one way. 
\item[$\Psi_{F'}$:] This is an additional \emph{destruction fidelity} term: it accounts for triples~$h \neq uvw$ whose selection 
would make two different available triples~$f,g \in F' \setminus\set{h}$ of the extension~$F'$ unavailable  
(it also takes freezing due to~$h \not\in F'$ into account, by allowing for~$uvw \in \set{f,g}$). 
\item[$\Lambda_{uvw,F,k-1}$:] 
This is the {\em creation fidelity} term: it accounts for available triples~$h \neq uvw$ in some extension~$F' \in W_{uvw,F,k-1}$, 
whose selection 
would make another available triple~$g \in F' \setminus \set{h}$ of~$F'$ unavailable  
(it also takes freezing due to~$h$ into account, by allowing for~$g=uvw$).  
\end{description}
These three fidelity terms are discussed in detail in \refS{sec:cl:fidelity}, where we prove the following bounds. 
\begin{claim}[Fidelity estimates]\label{cl:fidelity}
We have $\Pr(\neg \cE_i \cap \cG_i \text{ for some $0 \le i \le m_0$}) = o(n^{-\tau})$, 
where~$\cE_i$ denotes the event that, 
for all $f \in Q(i)$, $F \in \fF$, $0 \le k \le e_F-2$, and~$F' \in W_{f,F,k}(i)$, 
\begin{align}
\label{eq:up}
\Upsilon_f(i) & \le n^{\alpha}, \\
\label{eq:ps}
\Psi_{F'}(i) & \le n^{\alpha} , \\
\label{eq:lam}
\indic{k \ge 1}\Lambda_{f,F,k-1}(i) & \le n^{e_F-(1+k)+\alpha} . 
\end{align}
\end{claim}

\subsection{Trend hypothesis (supermartingale condition)}\label{sec:trend}
With the expressions for~$\E(\Delta X(i) \mid \cF_i)$ in hand,  
we now estimate the expected one-step changes of~$\Delta X^\pm(i)= X^{\pm}(i+1)-X^{\pm}(i)$. 
Recalling~$t=i/n^2$, by applying Taylor's theorem with remainder we obtain 
\begin{equation}\label{eq:Xpmi}
\begin{split}
\Delta X^\pm(i) &= \pm \biggsqpar{\Bigsqpar{|X(i+1)| - |X(i)|} - \Bigsqpar{\hx\bigpar{t+1/n^2}-\hx(t)}} - \Bigsqpar{e_X\bigpar{t+1/n^2} - e_X(t)} \\
& = \pm \biggsqpar{\Delta X(i) - \frac{\hx'(t)}{n^2}} - \frac{e'_X(t)}{n^2}  + O\biggpar{\sup_{s \in [0,m_0/n^2]}\frac{|\hx''(s)|+|e''_X(s)|}{n^4}} .
\end{split}
\end{equation}
The crux will be that~$\E(\Delta X(i) \mid \cF_i) \approx \hx'(t)/n^2$ holds (when~$\cF_i$ satisfies some natural conditions), 
so that the error term~$e'_X(t)/n^2$ enforces the supermartingale condition~$\E(\Delta X^{\pm}(i)\mid \cF_i) \le 0$.

In the following estimates of $\E(\Delta Y^{\pm}_{uv}(i) \mid \cF_i)$ we assume that~$uv \in E(i)$ and~$\cG_i \cap \cE_i$ hold.
To avoid clutter we introduce the abbreviation 
\begin{align*}
\tq(t) & :=  \sum_{ F \in \fF} \frac{ 6^2 e_F (e_F-1) ( 6t)^{e_F-2}}{ |{\rm Aut}(F)| } ,
\end{align*}
which conveniently satisfies~$q'(t) = -q(t) \tq(t)$ and 
\begin{equation}\label{eq:tqest}
\sum_{ F \in \fF} \frac{ \hat{w}_{F,e_F-2}(t)}{ \hat{q}(t)} = \frac{ \tilde{q}(t)}{ n^2} .
\end{equation}
By~\eqref{eq:heQY} we have~$\he_Q = \Theta(\he_Y)= o(1)$. 
Recalling~\eqref{eq:DYuv}, 
using~$e_{W_{F,e_{F}-2}}= e_Y = \Omega(n^{\alpha})$ we obtain  
\begin{equation}\label{eq:DYuv:est}
\begin{split}
\E(\D Y_{uv}(i)\mid \cF_{i}) & =  - \frac{\bigpar{1+O(\he_Y)}\hy}{\bigpar{1+O(\he_Q)}\hq} \cdot \Biggsqpar{ \bigpar{1+O(\he_Y)}2\hy + \sum_{F \in \fF} \hw_{F,e_{F}-2} + O\Bigpar{\sum_{F \in \fF} e_{W_{F,e_{F}-2}} + n^{\alpha}}}\\
& = - \bigpar{1+O(\he_Y)} \Biggsqpar{ \frac{12pq}{n} + \frac{p^2q \tq}{n}} + \frac{O\bigpar{e_Y|\fF|}}{n^2p} . 
\end{split}
\end{equation}
Noting that $-[12pq/n +p^2q \tq /n] = \hy'/n^2$, using~$\he_Y = e_Y/\hy$ and \refR{rem:FL} we arrive at 
\begin{equation}\label{eq:DYuv:est:2}
\begin{split}
\E(\D Y_{uv}(i)\mid \cF_{i}) & = \frac{\hy'}{n^2} + \frac{O\bigpar{e_Y\bigpar{1+p\tq+|\fF|}}}{n^2p} = \frac{\hy'}{n^2} + \frac{O(e_Y)}{n^2p} .
\end{split}
\end{equation}
Inspecting the generic estimate~\eqref{eq:Xpmi}, in order to establish the desired supermartingale condition $\E(\Delta Y_{uv}^{\pm}(i) \mid \cF_i) \le 0$ 
we thus require that the error function~$e_Y$ satisfies the \emph{variation equation}
\begin{equation}\label{eq:Y:cond}
e_Y'(t) \  \gg  \  \max\Biggcpar{\frac{ e_Y(t)}{p} , \; \sup_{s \in [0,m_0/n^2]}\frac{|\hy''(s)|+|e''_Y(s)|}{n^2} },  
\end{equation}
where we write $ f \gg g $ to denote that the ratio $f/g$ is sufficiently large.
(We shall verify~\eqref{eq:Y:cond} after deriving the corresponding variation equation for the~$e_{W_{F,k}}$ error functions.)

In the following estimates of $\E(\Delta W^{\pm}_{uvw,F,k}(i) \mid \cF_i)$ we assume that~$uvw \in Q(i)$ and~$\cG_i \cap \cE_i$ hold.  
Recalling that $\sum_{K \in \fF} e_{W_{K,e_{K}-2}}=O(e_Y)$, 
in view of~\eqref{eq:DWuvwFk} and $\he_Q = \Theta(\he_Y)$ 
we obtain 
\begin{equation*}
\begin{split}
\E(\D W_{uvw,F,k}(i) \mid \cF_{i}) & =  - \frac{\bigpar{\hw_{F,k}+O(e_{W_{F,k}})}}{\bigpar{1+O(\he_Y)}\hq} \cdot \Biggsqpar{ (e_F-1-k)\biggpar{3\hy + \sum_{K \in \fF} \hw_{K,e_{K}-2}} + O\bigpar{e_Y + n^{\alpha}}}\\
& \qquad + \frac{\indic{k \ge 1}}{\bigpar{1+O(\he_Y)}\hq} \cdot \Biggsqpar{\hw_{F,k-1}(e_F-k)  + O\bigpar{e_{W_{F,k-1}} + n^{e_F-(1+k)+\alpha}}} .
\end{split}
\end{equation*}
Note that $1/(1+O(\he_Y)) = 1+ O(\he_Y)$ and $\he_Y\hw_{F,j} = O(e_{W_{F,j}})$. 
Furthermore, $A \ge 2\ell \ge 2 e_F$ implies~$n^{e_F-(1+k)+\alpha} = O(e_{W_{F,k-1}})$. 
Using~$n^{\alpha} = O(e_Y)$ together with~\eqref{eq:tqest}, it follows that 
\begin{equation*}
\begin{split}
\E(\D W_{uvw,F,k}(i) \mid \cF_{i})
& =  - \bigpar{\hw_{F,k}+O(e_{W_{F,k}})}\Biggsqpar{ (e_F-1-k) \biggpar{\frac{18}{n^2p} + \frac{\tq}{n^2}} + \frac{O(e_Y)}{\hq}}\\
& \qquad + \indic{k \ge 1}\Biggsqpar{\frac{\hw_{F,k-1}(e_F-k)}{\hq}  + \frac{O\xpar{e_{W_{F,k-1}}}}{\hq}} .
\end{split}
\end{equation*}
Note that $\hy/\hq = \Theta(1)/(n^2p)$ implies $e_Y/\hq = \he_Y \cdot \hy/\hq \ll 1/(n^2p)$, so the first squared bracket is~$O(1)/(n^2p)$. 
Combining this 
with~$\hw_{F,k} e_Y/\hq = \he_Y\hw_{F,k} \cdot \hy/\hq = O(e_{W_{F,k}})/(n^2p)$ 
and~$e_{W_{F,k-1}} = O(e_{W_{F,k}} p^2n)$, 
in view of~$q=\Theta(1)$ it then follows that 
\begin{equation}\label{eq:DWuvwFk:Est3}
\begin{split}
\E(\D W_{uvw,F,k}(i) \mid \cF_{i})
 & =  - \Biggsqpar{\hw_{F,k}(e_F-1-k)\biggpar{ \frac{18}{n^2p} + \frac{\tq}{n^2}} + \frac{O\xpar{e_{W_{F,k}}}}{n^2p}}\\
& \qquad + \indic{k \ge 1}\Biggsqpar{\frac{\hw_{F,k-1}6(e_F-k)}{n^3p^3q}  + \frac{O\xpar{e_{W_{F,k}}}}{n^2p}} .
\end{split}
\end{equation}
Identifying the two main terms as~$\hw'_{uvw,F,k}/n^2$, 
we arrive at 
\begin{equation}\label{eq:DWuvwFk:Est4}
\begin{split}
\E(\D W_{uvw,F,k}(i) \mid \cF_{i})
 & =  \frac{\hw'_{uvw,F,k}}{n^2} + \frac{O\bigpar{e_{W_{F,k}}}}{n^2p} .
\end{split}
\end{equation}
To establish the supermartingale condition $\E(\Delta W_{uvw,F,k}^{\pm}(i) \mid \cF_i) \le 0$, in view of~\eqref{eq:Xpmi} 
we thus require that the error functions~$e_{W_{F,k}}$ satisfy the \emph{variation equations} 
\begin{equation}\label{eq:W:cond}
e_{W_{F,k}}'(t) \  \gg  \  \max\Biggcpar{\frac{ e_{W_{F,k}}(t) }{p} , \; \sup_{s \in [0,m_0/n^2]}\frac{|\hw_{F,k}''(s)|+|e''_{W_{F,k}}(s)|}{n^2} } .
\end{equation}

Finally, exploiting that all our implicit constants do \emph{not} depend on~$A$, 
for~$A \ge 2\ell$ large enough we readily satisfy 
the variation equations~\eqref{eq:Y:cond} and~\eqref{eq:W:cond} 
by combining the observation $e'_{X}=\Theta(A) \cdot e_{X}/p$ 
with the following claim (since~$1 \ll e_Y/p$ and~$n^{e_F-(2+k)} \ll e_{W_{F,k}}/p$, with room to spare). 
\begin{claim}[Derivative estimates]\label{cl:ekt}
For all $t \in [0,m_0/n^2]$, $F \in \fF$, and $0 \le k \le e_F-2$, 
\begin{align}
\label{eq:hYeY:cond}
\frac{|\hy'(t)|+|e'_Y(t)|}{n} + \frac{|\hy''(t)|+|e''_Y(t)|}{n^2} & = O(1)  , \\
\label{eq:hWeW:cond}
\frac{|\hw'_{F,k}(t)|+|e'_{W_{F,k}}(t)|}{n} + \frac{|\hw''_{F,k}(t)|+|e''_{W_{F,k}}(t)|}{n^2} & = O\bigpar{n^{e_F-(2+k)}}  .
\end{align}
\end{claim}
\begin{proof}[Proof-Sketch]
It is routine to check that~$|\hy'|,|\hy''| = O(n)$ and~$|\hw'_{F,k}|,|\hw''_{F,k}| = O(n^{e_F-(1+k)})$. 
Since~$1/p^{A+2} = O(n^{\beta(A+2)}) = o(n^{1-\alpha})$ by choice of~$\beta$, see~\eqref{def:alpha:beta}, 
noting that the first two derivatives satisfy~$e^{(j)}_{X}=\Theta((A/p)^j) \cdot e_{X}$ 
it then is straightforward to verify~\eqref{eq:hYeY:cond}--\eqref{eq:hWeW:cond} 
in view of~$e_Y=p^{-A}n^{\alpha}$ and~$e_{W_{F,k}} = O(e_Y) \cdot n^{e_F-(2+k)}$. 
\end{proof}

\subsection{Boundedness hypothesis (bounds on one-step changes)}\label{sec:bounded} 
Next, deferring the definition of the auxiliary event~$\cU_i$ (see \refCl{cl:bounded} below), 
we now establish the following 
bounds on the (expected) one-step changes~$\Delta X^\pm(i)$.  
Whenever~$uv \in E(i)$ and~$\cG_i \cap \cE_i \cap \cU_i$ hold, we have
\begin{align}
\label{eq:DYuv:E}
\E(|\Delta Y^{\pm}_{uv}(i)| \mid \cF_{i}) & = O\bigpar{n^{-1}}, \\
\label{eq:DYuv:max}
|\Delta Y^{\pm}_{uv}(i)| & = O\bigpar{n^{\alpha/2}}.
\end{align}
Whenever~$uvw \in Q(i)$ and~$\cG_i \cap \cE_i \cap \cU_i$ hold, for all $F \in \fF$ and $0 \le k \le e_F-2$, we have
\begin{align}
\label{eq:DWuvwFk:E}
\E(|\Delta W^{\pm}_{uvw,F,k}(i) |\mid \cF_{i}) & = O\bigpar{n^{e_F-(3+k)}}, \\
\label{eq:DWuvwFk:max}
|\Delta W^{\pm}_{uvw,F,k}(i) | & = O\bigpar{n^{e_F-(2+k)+\alpha/2}}.
\end{align}

On a first reading, the reader may perhaps wish to skip the below (conceptually not so illuminating) proofs of~\eqref{eq:DYuv:E}--\eqref{eq:DWuvwFk:max}, 
and continue directly with the large deviation estimates of~\refS{sec:martingale} 
(where it will be crucial that our upper bounds for~$\E(|\Delta X| \mid \cF_{i})$ are much smaller than for~$|\Delta X|$).

For bounds on the expected one-step changes~$\E(|\Delta X^\pm(i)| \mid \cF_{i})$ we shall exploit the fairly precise estimates from \refS{sec:trend}. 
Namely, using inequality~\eqref{eq:Y:cond}, note that the proof of~\eqref{eq:DYuv:est}--\eqref{eq:DYuv:est:2} shows that 
\begin{equation*}
\E(|\Delta Y_{uv}| \mid \cF_{i}) = \frac{O(pq+p^2q \tq)}{n} + \frac{O(e_Y)}{n^2p} = \frac{O(1)}{n} + \frac{O(e'_Y)}{n^2}.
\end{equation*}
Similarly, using~\eqref{eq:W:cond}, $\hw_{F,j} = O((p^3q n)^{e_F-(1+j)})$ and $q = \Theta(1)$, 
the proof of~\eqref{eq:DWuvwFk:Est3}--\eqref{eq:DWuvwFk:Est4} shows that
\begin{equation*}
\E(|\D W_{uvw,F,k}|\mid \cF_{i}) 
= \frac{O\bigpar{e_{W_{F,k}} + \hw_{F,k}}}{n^2p} +  \indic{k \ge 1}\frac{O\bigpar{\hw_{F,k-1}}}{n^3p^3q} 
= \frac{O(e'_{W_{F,k}})}{n^2} + O\bigpar{n^{e_F-(3+k)}} .
\end{equation*}
Using the generic estimate~\eqref{eq:Xpmi}, 
now inequalities~\eqref{eq:hYeY:cond}--\eqref{eq:hWeW:cond} imply the claimed bounds~\eqref{eq:DYuv:E} and~\eqref{eq:DWuvwFk:E}.

Turning to bounds on the one-step changes~$|\Delta X^\pm|$, we first record that the arguments above show 
\begin{align}
\label{eq:DYuv:max:red}
|\Delta Y^{\pm}_{uv}| & \le |\Delta Y_{uv}| +  O\bigpar{n^{-1}}, \\
\label{eq:DWuvwFk:max:red}
|\Delta W^{\pm}_{uvw,F,k}| & \le |\Delta W_{uvw,F,k}| + O\bigpar{n^{e_F-(3+k)}}. 
\end{align}
Recall that our freezing convention artificially enforces~$|\Delta Y_{uv}(i)|=0$ when~$uv \not\in E(i+1)$, 
and artificially enforces~$|\Delta W_{uvw,F,k}(i)|=0$ when~$uvw=e_{i+1}$ or~$uvw \not\in Q(i+1)$.  
Taking into account the changes of the available codegree~$|Y_{uv}|$ and the extensions variables~$|W_{uvw,F,k}|$ discussed in \refS{sec:ECh}, 
it follows that 
\begin{align}
\label{eq:DYuv:est:max}
|\D Y_{uv}| & \le \indic{|uv \setminus f| \ge 1} \cdot O\bigpar{\Pi_{uv,e_{i+1}}} ,\\
\label{eq:DWuvwFk:est:max}
|\D W_{uvw,F,k}| &\le \indic{e_{i+1} \neq uvw} \cdot O\bigpar{\Pi_{uvw,F,k,e_{i+1}} + \Phi_{uvw,F,k,e_{i+1}}+\indic{k \ge 1}\Pi_{uvw,F,k-1,e_{i+1}}} ,
\end{align}
where the three boundedness parameters correspond to the following effects:
\begin{description}
\itemsep 0.125em \parskip 0em  \partopsep=0pt \parsep 0em
\item[$\Pi_{uv,e_{i+1}}$:] 
This term accounts for triples~$uvz \in Q(i)$ with~$z \in Y_{uv}(i)$ that become unavailable 
due to the addition of~$e_{i+1}$ to~$\cH_i$. 
\item[$\Pi_{uvw,F,j,e_{i+1}}$:] 
This term accounts for extensions~$F'$ which leave~$W_{uvw,F,j}$ because~$e_{i+1}$ equals an available triple in $F' \setminus \set{uvw}$. 
(Note that in the case~$j=k-1 \ge 0$ this term also accounts for all extensions~$F' \in W_{uvw,F,k-1}$ which can possibly move into~$W_{uvw,F,k}$.) 
\item[$\Phi_{uvw,F,k,e_{i+1}}$:] 
This term accounts for extensions~$F'$ which leave~$W_{uvw,F,k}$ because an available triple in~$F' \setminus \set{uvw}$ becomes unavailable 
due to the addition of~$e_{i+1} \not\in F'$ to~$\cH_i$. 
\end{description}
These parameters are discussed in more detail in \refS{sec:cl:bounded}, where we prove the following bounds. 
\begin{claim}[Boundedness estimates]\label{cl:bounded}
We have $\Pr(\neg \cU_i \cap \cG_i \text{ for some $0 \le i \le m_0$}) = o(n^{-\tau})$, 
where~$\cU_i$ denotes the event that, 
for all $uv \in \binom{[n]}{2}$, $f,g \in Q(i)$, $F \in \fF$, and $0 \le k \le e_F-2$, 
\begin{align}
\label{eq:Phi:Est}
\indic{|uv \setminus f| \ge 1}\Pi_{uv,f}(i) & \le n^{\alpha/2},\\
\label{eq:Delta:Est}
\indic{f \neq g}\Pi_{f,F,k,g}(i) & \le n^{e_F-(2+k)+\alpha/2 - \indic{k< e_F-2}} , \\
\label{eq:Pi:Est}
\indic{f \neq g}\Phi_{f,F,k,g}(i) &\le n^{e_F-(2+k)+\alpha/2}.
\end{align}
\end{claim}
\noindent
Finally, combining the event~$\cU_i$ with~\eqref{eq:DYuv:max:red}--\eqref{eq:DWuvwFk:est:max}, 
now the claimed bounds~\eqref{eq:DYuv:max} and~\eqref{eq:DWuvwFk:max} follow readily 
(since~$k':= k-1$ satisfies~$-(2+k') - \indic{k' < e_F-2} = -(2+k)$ when~$k \le e_F-2$).

\subsection{Supermartingale estimates}\label{sec:martingale}
We are now ready to bound~$\Pr(\neg\cG_{m_0})$. 
To this end we focus on the first step in which~$|X(i)| = \hat{x}(t) \pm e_X(t)$ (and thus~$X^\pm(i) \le 0$) 
is violated for some variable~$X$ of form~$Y_{uv}$ or~$W_{uvw,F,k}$. 
Our main tool for bounding the probability of each such `first bad event' is the 
following simple consequence of Freedman's martingale inequality~\cite{F} (see also~\cite{TBDI}):    
it intuitively improves the 
Azuma--Hoeffding inequality 
when the expected one-step changes are much smaller than the worst case ones,  
 i.e., when $\E(|\Delta S(i)| \mid \mc{F}_{i}) \ll \max |\Delta S(i)|$. 
\begin{lemma}\label{lem:freedman}
Let $(S(i))_{i \ge 0}$ be a supermartingale with respect to the filtration $\mc{F}=(\mc{F}_i)_{i \ge 0}$. 
Writing $\Delta S(i) := S(i+1)-S(i)$, suppose that~$\max_{i \ge 0}|\Delta S(i)| \le C$ and~$\sum_{i \ge 0} \E(|\Delta S(i)| \mid \mc{F}_{i}) \le V$.  
Then, for any $z >0$, 
\begin{equation}
\label{eq:Free}
\mb{P}\left(S(i) \ge S(0) + z \text{ for some $i \ge 0$}\right)
\le \exp \left\{-\frac{z^2}{2C(V+z)} \right\}. 
\end{equation}
\end{lemma}

Turning to the details, define~$\cG^+_i := \cG_i \cap \cE_i \cap \cU_i$. 
Note that initially~$|Y_{uv}(0)|=n-2$ and~$|W_{uvw,F,k}(0)|=\indic{k=0}N_{uvw,F}$ hold, 
which in view of~\eqref{eq:NuvwF} and $X^\pm(i) = \pm\bigsqpar{|X(i)| - \hx(t)} - e_X(t)$ gives
\begin{gather}
\label{eq:YuvLinit}
Y^\pm_{uv}(0) 
\le O(1) -e_Y(0) \le -e_Y(0)/2 , \\ 
\label{eq:WFk:init}
W^{\pm}_{uvw,F,k}(0) 
\le \indic{k = 0}O(n^{e_F-2}) - e_{W_{F,k}}(0) \le - e_{W_{F,k}}(0)/2 , 
\end{gather}
implying that~$\cG_0$ holds deterministically (with room to spare). 
Using Claims~\ref{cl:fidelity} and~\ref{cl:bounded}, we infer 
\begin{equation*}
\Pr(\neg\cG_{m_0}) \le \Pr\bigpar{\neg \cG_{i} \cap \cG^+_{i-1}  \text{ for some $1 \le i \le m_0$}} + o(n^{-\tau}) .
\end{equation*}
Since the two estimates~$X^\pm(i) \le 0$ together imply~$|X(i)| = \hat{x}(t) \pm e_X(t)$, it follows that 
\begin{equation*}
\neg \cG_{i} \cap \cG_{i-1} \subseteq \bigcup_{uv \in E(i)} \Bigcpar{\max_{\sigma \in \set{+,-}}Y^{\sigma}_{uv}(i) \ge 0} \cup  \bigcup_{uvw \in Q(i)}\bigcup_{F \in \fF}\bigcup_{0 \le k \le e_F-2}\Bigcpar{\max_{\sigma \in \set{+,-}}W^{\sigma}_{uvw,F,k}(i) \ge 0} .
\end{equation*}
(To clarify: here we tacitly used that the bounds on~$|Q(i)|$ follow from the bounds on~$|Y_{uv}(i)|$. 
Moreover, we used that our freezing convention does not affect~$|Y_{uv}(i)|$ as long as~$uv \in E(i)$, and also does not affect~$|W_{uvw,F,k}(i)|$ as long as~$uvw \in Q(i)$.)   
Let the stopping time~$T_{uv}$ be the minimum of~$m_0$ and the first step~$i \ge 0$ where~$uv \not\in E(i)$ or~$\neg\cG^+_i$ holds. 
Similarly, let the stopping time~$T_{uvw}$ be the minimum of~$m_0$ and the first step~$i \ge 0$ where~$uvw \not\in Q(i)$ or~$\neg\cG^+_i$ holds. 
Writing~$i \wedge T := \min\set{i,T}$, as usual, it follows that 
\begin{equation*}
\begin{split}
&\Pr\bigpar{\neg \cG_{i} \cap \cG^+_{i-1}  \text{ for some $1 \le i \le m_0$}} \le \sum_{uv \in \binom{[n]}{2}}\sum_{\sigma \in \set{+,-}}\Pr\bigpar{Y^{\sigma}_{uv}(i \wedge T_{uv}) \ge 0 \text{ for some $i \ge 0$}}\\
& \qquad \qquad \qquad \qquad \quad 
+ \sum_{uvw \in \binom{[n]}{3}}\sum_{F \in \fF}\sum_{0 \le k \le e_F-2} \sum_{\sigma \in \set{+,-}}\Pr\bigpar{W^{\sigma}_{uvw,F,k}(i \wedge T_{uvw}) \ge 0 \text{ for some $i \ge 0$}}. 
\end{split}
\end{equation*}
The crux is that, by~\eqref{eq:YuvLinit} and the calculations from Sections~\ref{sec:trend}--\ref{sec:bounded} (as~$i < T_{uv}$ implies that both~$uv \in E(i)$ and~$\cG_i^+=\cG_i \cap \cE_i \cap \cU_i$ hold), the sequence~$S(i) := Y^{\sigma}_{uv}(i \wedge T_{uv})$ is a supermartingale with $S(0)\le -e_Y(0)/2$, 
to which we can apply \refL{lem:freedman} with~$C = O(n^{\alpha/2})$ and~$V= m_0 \cdot O(n^{-1}) = O(n)$. 
Invoking inequality~\eqref{eq:Free} with $z=e_Y(0)/2 = \Theta( n^{\alpha})$, in view of~$\alpha > 2/3$ we obtain 
\begin{equation*}
\Pr\bigpar{Y^{\sigma}_{uv}(i \wedge T_{uv}) \ge 0 \text{ for some $i \ge 0$}} \le \exp\biggcpar{-\frac{\Theta(n^{2\alpha})}{O(n^{\alpha/2}) \cdot O(n + n^{\alpha})}} 
\le n^{-\omega(1)} .
\end{equation*}
Similarly, 
the sequence~$S(i) := W^{\sigma}_{uvw,F,k}(i \wedge T_{uvw})$ is a supermartingale with $S(0) \le -e_{W_{F,k}}(0)/2$, 
to which we can apply \refL{lem:freedman} with~$C = O\bigpar{n^{e_F-(2+k)+\alpha/2}}$ and~$V= m_0 \cdot O\bigpar{n^{e_F-(3+k)}} = O(n^{e_F-(2+k)+1})$. 
Invoking inequality~\eqref{eq:Free} with $z=e_{W_{F,k}}(0)/2 = \Theta( n^{e_F-(2+k)+\alpha})$,  
it follows that 
\begin{equation*}
\Pr\bigpar{W^{\sigma}_{uvw,F,k}(i \wedge T_{uvw}) \ge 0 \text{ for some $i \ge 0$}} \le 
\exp\biggcpar{-\Omega\biggpar{\frac{n^{2\alpha}}{n^{\alpha/2}(n + n^{\alpha})}}}  
\le n^{-\omega(1)} .
\end{equation*}
Assuming Claims~\ref{cl:fidelity} and~\ref{cl:bounded} (whose proofs are given in \refS{sec:crude}), 
in view of~$|\fF|=O(1)$ this completes the proof of~$\Pr(\neg\cG_{m_0}) = o(n^{-\tau})$ and thus \refT{thm:main}.

\subsection{Auxiliary results: Crude extension estimates}\label{sec:crude} 
In this final subsection we prove Claims~\ref{cl:fidelity} and~\ref{cl:bounded} (and thus complete the proof of \refT{thm:main}), 
by exploiting some crude estimates on hypergraph extensions that hold throughout the high-girth triple-process. 
We shall formally think of these extensions in terms of injective functions from
the vertex-set of some fixed (and bounded) hypergraph $H$ to the vertex-set~$[n]$ of~$\cH(i)$.  As usual, 
such an injection $\psi$ lifts to a map on sets and sets of sets using the abbreviations 
$\psi(xyz)=\psi(x)\psi(y)\psi(z)$ and $\psi(E) = \bigcup_{f \in E} \psi(f)$.   For hypergraphs $ G \subset H $ and injection $ \rho: V(G) \to [n]$
define
\begin{align}
N_{\rho,G,H}(i) &:= \bigabs{\bigcpar{ \text{injection } \psi : \; V(H) \to [n] \text{ with } \psi{\big|}_{V(G)} \equiv \rho \text{ and } \psi\bigpar{H \setminus G} \subseteq \cH_i}} .
\end{align}
In words, $N_{\rho,G,H}(i)$ counts the number of (labeled) copies of~$H$ in~$\cH_i \cup \rho(G)$ 
which contain the distinguished (labeled) copy~$\rho(G)$ of~$G$. 
\begin{theorem}\label{thm:ext}
Let~$\cT=\cT_\ell$ denote the collection of all hypergraph tuples~$(G,H)$ 
with~$\max\set{v_H,e_H} \le 2\ell$ and~$G \subset H$. 
Let~$\cC_i$ denote the event that, for all $(G,H) \in \cT$ and all injections $\rho : V(G) \to [n]$, 
\begin{equation}\label{eq:thm:ext}
N_{\rho,G,H}(i) \le n^{\alpha/9} \cdot \max_{G \subseteq J \subseteq H} n^{(v_H-e_H) + (e_J-v_J)} .
\end{equation}
Then $\Pr(\neg\cC_i \cap \cG_i \text{ for some $0 \le i \le m_0$}) = o(n^{-\tau})$.  
\end{theorem}
\noindent
Recalling \refS{sec:pseudo-random}, inequality~\eqref{eq:thm:ext} can best be understood by thinking of a 
random hypergraph with~$n$~vertices and triple-probability~$\pi = n^{-1+O(\beta)}$, 
the heuristic idea being that the `most difficult root' $G \subseteq J \subseteq H$ matters  
(as $n^{v_H-v_J}\pi^{e_H-e_J} = n^{O(\ell \beta)} \cdot n^{(v_H-e_H) + (e_J-v_J)}$ corresponds, up to constant factors, to the expected number of copies of~$H$ containing a fixed copy of~$J$; 
furthermore, $\ell\beta = O(\alpha)$ by choice~\eqref{def:alpha:beta} of~$\beta$). 
The following short moment based proof is inspired by 
arguments of {\v{S}}ileikis and Warnke~\cite{SW}. 
\begin{proof}[Proof of \refT{thm:ext}]
When~$\cG_i$ holds, with $i \le m_0$, then in every step~$j \le i$ there are~$|Q(j)| \ge |Q(i)| \ge \hq(t)/2 \ge n^{3-4\beta}$ available triples. 
For any set~$T \subseteq \binom{[n]}{3}$ of triples, 
a straightforward adaptation of the proof of~\cite[Lemma~4.1]{BK} 
(which proceeds by taking a union bound over all steps where the triples of~$T$ could appear) 
thus gives 
\begin{equation}\label{eq:lem:edges}
\max_{0 \le i \le m_0}\Pr(\cG_i \text{ and } T \subseteq \cH_i) \le m_0^{|T|} \cdot \bigpar{1/n^{3-4\beta}}^{|T|} \le \pi^{|T|} \quad \text{ with } \quad \pi := n^{-1+4\beta}.  
\end{equation}

Fix~$0 \le i \le m_0$, $(G,H) \in \cT$, and an injection~$\rho : V(G) \to [n]$. 
By~\eqref{eq:lem:edges}, for any integer $s \ge 1$ we have 
\begin{equation}\label{eq:lem:central}
\E \bigpar{\indic{\cG_i}N_{\rho,G,H}(i)}^s = \sum_{(\psi_1, \ldots, \psi_s)} \E\Bigpar{\indic{\cG_i} \prod_{j \in [s]} \indic{\psi_j\xpar{ H \setminus G} \subseteq  \cH_i}} \le \sum_{(\psi_1, \ldots, \psi_s)} \pi^{|\bigcup_{j \in [s]} E_j|} ,
\end{equation}
where~$E_j := \psi_j(H \setminus G)$ and the sum is over all $s$-tuples 
of injections~$\psi_j: V(H) \to [n]$ with~$\psi_j{\big|}_{V(G)}\equiv\rho$. 
Note that $E_j \setminus \bigcup_{k \in [j-1]} E_k \cong H \setminus J_j$ and $E_j \cap \bigcup_{k \in [j-1]} E_k \cong J_j \setminus G$ for some~$G \subseteq J_j \subseteq H$. 
Taking all (of the boundedly many) possible types of `overlaps'~$(J_1, \ldots, J_{m})$ into account, 
it follows that 
\begin{equation}\label{eq:lem:central3}
\E \bigpar{\indic{\cG_i}N_{\rho,G,H}(i)}^s \le \sum_{(J_1, \ldots, J_s)} D_s \prod_{j \in [s]} n^{v_H-v_{J_j}}\pi^{e_H-e_{J_j}} \le C_s \Bigpar{\max_{G \subseteq J \subseteq H} n^{v_H-v_J} \pi^{e_H-e_J}}^s ,
\end{equation}
where~$C_s,D_s>0$ may depend on~$s,G,H$. 
Using~$s := \ceil{(4+2\ell+\tau)/\beta}$, for $n \ge n_0(C_s)$ it follows that 
\begin{equation*}
\Pr\Bigpar{ N_{\rho,G,H}(i) \ge n^{\beta} \max_{G \subseteq J \subseteq H} n^{v_H-v_J} \pi^{e_H-e_J} \text{ and } \cG_i } 
\le C_s \cdot n^{-\beta s} \le n^{-(3+2\ell+\tau)} . 
\end{equation*}
This completes the proof by a standard union bound argument (that accounts for all possible~$i,G,H,\rho$), 
since~$n^{\beta}\pi^{e_H-e_J} \le n^{\alpha/9 + (e_J-e_H)}$ follows from~$\pi = n^{-1+4\beta}$, $e_H-e_J \le 2\ell$,  
and the definition~\eqref{def:alpha:beta} of~$\beta$. 
\end{proof}

In the remainder of this section we use \refT{thm:ext} to prove \refCl{cl:fidelity} and~\ref{cl:bounded}.   We begin with some
precise notation for counting extensions.  
To  account for `partial' copies of~$F$ which extend some fixed set of triples~$R$ and vertices~$S$, 
we introduce 
\begin{align}
\label{def:Gamma:RFkS}
\Gamma_{R,S,F,k}(i) &:= \bigcpar{F' \in \bF_{F}: \ R \subseteq F', \ |\cH_i \cap F' \setminus R| \ge k, \text{ and } S \subseteq V(F')},
\end{align}
where the set~$\bF_F$ of all $F$--copies is defined as in \refS{sec:variables}. 
Similarly, to account for `overlapping' partial copies of~$F$ and~$K$ extending certain triples 
(note 
that~$|K' \cap Q(i) \setminus \set{g}|=1$ below), we introduce 
\begin{align}
\begin{split}
\label{def:W:fFkgK}
\cW_{f,F,k,g,K}(i) &:= \bigset{(F',K',h) \in W_{f,F,k}(i) \times W_{g,K,e_K-2}(i) \times Q(i): \  \\
& \qquad \qquad \qquad \qquad F' \neq K'  \text{ and } \set{h}= F'  \cap K' \cap Q(i) \setminus \set{f,g}},
\end{split}\\
\label{def:W:fFkK}
\cW^+_{f,F,k,K}(i) &:= \bigset{(g,F',K',h): \ (F',K',h) \in \cW_{f,F,k,g,K}(i) \text{ and } g \in F' \cap Q(i) \setminus \set{h}} ,
\end{align}
where the extension-sets~$W_{uvw,F,k}(i)$ are defined as in \refS{sec:variables}. 
Finally, we mention that our applications of~\eqref{eq:thm:ext} will hinge on the following 
simple 
consequence of the minimality of the obstructions~$F \in \fF^{+}$. 
\begin{lemma}\label{lem:minimality}
Let~$F \in \fF^+$ and~$G \subseteq F$ with~$v_G \ge 2$. 
Then~$e_G \le v_G -(2 + \indic{v_G \ge 4 \text{ and } G \neq F})$. 
\end{lemma}
%

\subsubsection{Boundedness estimates: Proof of \refCl{cl:bounded}}\label{sec:cl:bounded}
Reinspecting \refS{sec:bounded},
we now formally define the random variables treated by \refCl{cl:bounded}:
\begin{align}
\label{eq:53}
\Pi_{uv,f}(i) & := \sum_{F \in \fF^+}|\Gamma_{\set{f},uv,F,e_{F}-2}(i)| , \\ 
\label{eq:54}
\Pi_{f,F,k,g}(i) & := |\Gamma_{\set{f,g},\emptyset,F,k}(i)| , \\
\label{eq:55}
\Phi_{f,F,k,g}(i) & := \sum_{K \in \fF^+}|\cW_{f,F,k,g,K}(i)|. 
\end{align}
Our plan is to show that the event~$\cC_i$ from \refT{thm:ext} implies the estimates~\eqref{eq:Phi:Est}--\eqref{eq:Pi:Est} of \refCl{cl:bounded}.

\noindent
{\bf Proof of \eqref{eq:Phi:Est}.}  
Assuming $|uv \setminus f| \ge 1$, it suffices to show that each term in~\eqref{eq:53} is~$o(n^{\alpha/2})$. 
Let~$ F \in \fF^+$.  
Consider a hypergraph~$ H $ on vertex-set $ V(F)$ of the form $H = F \setminus \{f_1, f_2\} $ and a hypergraph~$G$ defined to be the empty hypergraph on vertex-set $ V(G) :=  f_1 \cup u_1v_1$, where $|u_1v_1 \setminus f_1| \ge 1$.
We are interested in the number of injections of~$V(H)$ to~$[n]$ that map~$f_1$ onto the triple~$f$, map~$u_1v_1$ onto the pair~$uv$, and map the triples of~$H$ to triples of~$\cH_i$. 

We now apply Theorem~\ref{thm:ext}.  Note that $ e_H = e_F-2=v_H - 4$.
Consider $ G \subseteq J \subseteq H $.  
If~$V(J) \subsetneq V(H)$ then by \refL{lem:minimality} (as~$v_J \ge v_G \ge 4$) 
the subhypergraph of $F$ induced on $ V(J)$ has at most~$ v_J-3$ triples.  Therefore, as $f_1 \in F$
and $f_1 \notin J$, we have $ e_J \le v_J- 4 $ and thus
\[ (v_H - e_H) - ( v_J - e_J) \le  4  - 4 = 0 .\]
If~$V(J) = V(H)$ then we trivially have $(v_H - e_H) - ( v_J - e_J) = e_J - e_H \le 0$, too. 
Thus, for any such pair~$(G,H)$, by inequality~\eqref{eq:thm:ext} there are at most~$ n^{\alpha/9} $ relevant injections. 
By summing over the bounded number of such~$(G,H)$, we readily infer $\indic{|uv \setminus f| \ge 1}|\Gamma_{\set{f},uv,F,e_{F}-2}(i)| =O(n^{\alpha/9})$, completing the proof of~\eqref{eq:Phi:Est}.

\noindent
{\bf Proof of \eqref{eq:Delta:Est}.} 
Assuming $f \neq g$, it suffices to bound~\eqref{eq:54}. 
Let~$ F \in \fF $.
Consider a hypergraph $ H $ on vertex-set $ V(F)$ with $ H \subset F $ and $ |H| =k $ and a hypergraph~$ G$ defined to be the empty hypergraph on vertex-set~$ V(G) :=  f_1 \cup f_2 $, where~$ f_1,f_2 \in F \setminus H$ and~$f_1 \neq f_2$.
We are interested in the number of injections of~$ V(H)$ to~$[n]$ that map~$f_1$ onto~$f$, map~$ f_2 $ onto the triple~$g$, and map the triples of~$H$ to triples of~$\cH_i$. 

We again apply Theorem~\ref{thm:ext}. 
Consider $ G \subseteq J \subseteq H $.  It suffices to show
\begin{equation}
\label{eq:suffices}
(v_H - e_H) - ( v_J - e_J) \le e_F - (2+k) - \indic{ k < e_F -2}.
\end{equation}
If~$ V(J) =V(H)$ then the trivial estimate~$(v_H - e_H) - ( v_J - e_J) = e_J - e_H \le 0$ implies~\eqref{eq:suffices}. 
If~$V(J) \subsetneq V(H)$, then by \refL{lem:minimality} (as~$f_1 \neq f_2$ implies~$v_J \ge v_G \ge 4$) 
the subhypergraph of~$F$ induced on~$ V(J)$ has at most~$v_J-3$ triples.  In that case, as~$f_1, f_2 \in F$
and~$f_1, f_2 \notin J$, we have~$ e_J \le v_J- 5 $ and thus 
\[ (v_H - e_H) - ( v_J - e_J) \le  v_H - k - 5 = e_F - k -3, \]
establishing~\eqref{eq:suffices}.
Now inequality~\eqref{eq:Delta:Est} follows by similar reasoning as for~\eqref{eq:Phi:Est} above.

\noindent
{\bf Proof of \eqref{eq:Pi:Est}.} 
In view of~\eqref{eq:55}, this is equivalent to a bound on $|\cW_{f,F,k,g,K}|$.  As this bound will
also be needed in the proof of the fidelity estimates given in the next subsection, we now state general bounds
on~$|\cW_{f,F,k,g,K}|$ and~$|\cW^+_{f,F,k,g,K}|$.   Note that, by~\eqref{eq:55}, inequality~\eqref{eq:Pi:Est} follows immediately from~\eqref{eq:overlap} below. 
\begin{lemma}\label{lem:degree1}
The event~$\cC_i$ from \refT{thm:ext} implies the following for~$n \ge n_0(\ell,\tau,\alpha,\beta,A)$. 
For all~$F,K \in \fF^+ $ with $(F,K) \notin (\fF^+ \setminus \fF)^2 $ and $f,g \in Q(i)$ and  $0 \le k \le e_F-2$, 
\begin{align}\label{eq:overlap}
|\cW_{f,F,k,g,K}(i)| & \le n^{e_F-(2+k)+\alpha/8} .
\end{align}
Furthermore, for all~$F,K \in \fF $ and $f \in Q(i)$ and $0 \le k \le e_F-3$, 
\begin{align}\label{eq:overlap2}
|\cW^+_{f,F,k,K}(i)| & \le n^{e_F-(2+k)+\alpha/8} .
\end{align}
\end{lemma}
\begin{proof}
We begin with the bound~\eqref{eq:overlap} for~$|\cW_{f,F,k,g,K}|$. 
Let~$ F, K \in \fF^+ $ such that at least one of~$ F,K $ is in~$\fF$. 
Let~$ F' \subset F $ with~$|F'| = k$ and~$V(F')=V(F)$. 
Let~$K' \subset K$ with~$|K'| = e_K-2$ and~$V(K')=V(K)$. 
Consider~$H = F' \cup K'$ such that there are triples~$ f,h,g \in \binom{V(H)}{3} $ such~that
\begin{equation}
f,g,h \not\in H,
\qquad 
K' \cup \{g,h\} = K,
\quad \text{ and } \quad
F' \cup \{ f,h \} \subseteq F.
\end{equation}
We allow~$f=g $, but we impose the conditions~$f \neq h$ and~$g \neq h$. 
Our aim is to bound the 
number of embeddings of~$H$ into~$\cH_i$ with the property that~$f$ and~$g$ map onto specified 
available triples of~$\cH_i$. We thus define~$G$ to be the 
subhypergraph of~$H$ induced on vertex-set~$V(G) := f  \cup g $. 

We apply Theorem~\ref{thm:ext}.  Consider $ G \subseteq J \subseteq H $.  
Let~$A := J \cap F'$ with vertex-set $V(A) := V(J) \cap V(F')$. 
Let~$B := K' \cap ( J \cup F' ) $ with vertex-set~$V(B) :=  V(K') \cap \left(  V(J) \cup V(F') \right)$.  
Note that $e_{F' \cap K'}+ e_J=e_A+e_B$. 
Combined with a similar decomposition of $v_A+v_B$, 
it is routine to see that 
\[   (v_H - v_J) - ( e_H - e_J) 
= (v_{F'}+v_{K'}) - (v_A+v_B)- ( e_{F'} + e_{K'}) + (e_A + e_B) .
\]
Note that~$f \subseteq V(A) $ and~$f \in F \setminus F'$. 
  Therefore, as $A \subseteq F \in \fF^+$,  we have~$ e_A -v_A \le - 3 - \indic{ v_A \ge 4 }$ by \refL{lem:minimality}.  
Note further that~$ g, h  \subseteq V(B)$ and~$g,h \in K \setminus K'$. 
As~$B \subseteq K \in \fF^+$, we infer $e_B-v_B \le -4 - \indic{ B \neq K' }$ 
by \refL{lem:minimality} (as~$g \neq h$ implies~$v_B \ge |g \cup h| \ge 4$). 
Noting~$v_{F'}-e_{F'}=e_F+2-k$ and $v_{K'}-e_{K'}=4$, 
it follows that 
\[    (v_H - v_J) - ( e_H - e_J) 
 \le e_F -( 1+ k)  - \indic{ v_A \ge 4 } -  \indic{ B \neq K'} .  \]
By similar reasoning as for~\eqref{eq:Delta:Est}--\eqref{eq:Pi:Est} above, we can apply Theorem~\ref{thm:ext} to complete the proof of (\ref{eq:overlap}), except in the case that~$v_A = 3$ and~$B = K' $ both hold.
We now show that this situation is not possible.

Assume for the sake of contradiction that~$v_A = 3$ and~$B = K' $ both hold. 
We first observe that if~$K$ is the diamond then we immediately have a contradiction. 
Indeed, in this case~$F \in \fF$ by assumption, so that~$|g \cap h|=2$ and~$|f \cap h| \le 1$, 
which in turn implies $v_A \ge |f \cup (g \cap h)| \ge 3+2-1>3$.
So we henceforth assume~$K \in \fF $.   
Note that we have~$V(A) = f$ and~$V(B) = V(K')$, so~$V(J) = f \cup (V(K') \setminus V(F'))$ follows.  
We consider two cases.  
First, if~$V(K') \subseteq V(F')$ then one can show that there exists a triple~$g' \in K \setminus \set{f}$ 
which satisfies~$g' \subseteq V(J) \cap V(K')$, contradicting~$V(A) =f$.
(One can show existence of~$g'$ as follows. 
By definition of~$\fF$ there must be at least one triple~$g' \in K$ that does not appear in~$F \in \fF$, 
which by construction satisfies~$g' \neq f$. 
If~$g'=g$ then~$g' \subseteq V(G) \subseteq V(J)$; 
otherwise~$g' \in K \setminus (F \cup \set{g}) \subseteq K'$ implies~$g' \in J$ due to~$B=K'$, establishing the claim.)   
Second, if~$ V(K') \not\subseteq V(F')$ then one can show that there exists a triple~$h' \in K'$ 
that intersects both~$X:=V(K') \setminus V(F')$ and~$Y:=(V(K') \cap V(F')) \setminus f$, 
which is not in~$B=(K' \cap F') \cup (K' \cap J)$ due to~$V(J)=f \cup (V(K') \setminus V(F'))$, 
contradicting~$B=K'$.  
(One can show existence of~$h'$ as follows. 
First, using~$K \in \fF$ it~is~an easy exercise to verify that, for any vertex-set~$W \subseteq V(K)$ with~$|W| \le 3$ and  
any partition of~$V(K)\setminus W$ into two nonempty parts~$X,Y$, there is at least one triple~$h' \in K$ that intersects both~$X$ and~$Y$. 
Second, applying this with~$W:=V(K') \cap f$ and~$X,Y$ as defined above, 
it remains to verify that no triple in~$K \setminus K'=\set{g,h}$ intersects both~$X$ and~$Y$. 
This is trivial for~$h \subseteq V(F')$. 
For~$g$ this is also trivial if~$g \cap V(F') = \emptyset$. 
Otherwise~$g \subseteq V(G) \subseteq V(J)$ and~$v_A=3$ imply~$g \cap V(F') \subseteq f$, 
so that~$|g \cap V(F')| \ge 1$ enforces~$|g \cap f| \ge 1$, 
which due to~$f \cap (X \cup Y) = \emptyset$ establishes the claim.)

Finally, we turn to the bound~\eqref{eq:overlap2} for~$|\cW^+_{f,F,k,K}(i)|$. 
Let $ F , K \in \fF$.  Let~$ F' \subset F $ with~$ |F'| = k $  and~$V(F')=V(F)$.  
Let~$K' \subset K $ with~$|K'| = e_K-2 $ and~$V(K')=V(K)$. 
Let $H = F' \cup K'$ such that there are triples $ f,h,g \in \binom{V(F')}{3} $ such~that
\begin{equation}
f,g,h \not\in H, 
\qquad 
 K' \cup \{g,h\} = K, 
\quad \text{ and } \quad
 F' \cup \{ f,g,h \} \subseteq F.
\end{equation}
We allow~$f=g $, but we impose the conditions~$f \neq h$ and~$g \neq h$. 
Our aim is to bound the 
number of embeddings of~$H$ into~$\cH_i$ with the property that~$f$ maps onto a specified 
available triple of~$\cH_i$.  We thus define~$G$ to be the empty hypergraph on vertex-set~$V(G):= f$. 
We can now follow the argument for~\eqref{eq:overlap} from the preceding
paragraphs (essentially verbatim, exploiting in the final contradiction arguments 
that here~$g \in K \cap F$ ensures~$g' \neq g$, and that here~$g \subseteq V(F')$ implies~$g \cap X=\emptyset$) 
to establish~\eqref{eq:overlap2}.
\end{proof}

\subsubsection{Fidelity estimates: Proof of \refCl{cl:fidelity}}\label{sec:cl:fidelity}
Reinspecting \refS{sec:ECh}, the random variables treated by \refCl{cl:fidelity} 
satisfy 
\begin{align}
\label{def:upsilon:f}
\Upsilon_f(i) & \le \sum_{L,K\in \fF} |\cW_{f,L,e_L-2,f,K}(i)|, \\
\label{def:Psi:f}
\Psi_{F'}(i) & := \sum_{\substack{f,g \in F' \cap Q(i):\\f \neq g}} \; \sum_{L,K \in \fF^+} |\cW_{f,L,e_{L}-2,g,K}(i)|, \\
\label{def:Lambda:fFk1}
\Lambda_{f,F,k-1}(i) & := \sum_{K \in \fF} |\cW^+_{f,F,k-1,K}(i)| .
\end{align}
The estimates in \refCl{cl:fidelity} now follow from Lemma~\ref{lem:degree1} 
(since~$F' \in \fF$ implies~$|f \cap g| \le 1$ in~\eqref{def:Psi:f}, it is not difficult to see that the `two diamonds' case~$L,K \in \fF\setminus \fF^+$ only contributes~$O(1)$ to~$\Psi_{F'}$).

\section{Concluding remarks}\label{sec:conclusion} 
It would be interesting to further explore the high-girth triple-process
and its connection with the random triangle removal process. 
This removal process was originally formulated by Bollob\'as and Erd\H{o}s,
who conjectured that at the end~$\Theta( n^{3/2} ) $ edges remain (recall 
that the remaining edges of the removal process correspond to the terminal edge-set~$E(m)$ in the high-girth process with~$\ell=4$, cf.~\refS{sec:high-girth}). 
Bohman, Frieze, and Lubetzky~\cite{BFL} proved an approximate version of this conjecture, showing that typically~$n^{3/2+o(1)}$ edges remain. 
It is natural to conjecture that the same result also holds for the high-girth process, 
since obstructions on more than $4$~vertices have a negligible impact during the early evolution
 (see \refR{rem:main} and \refT{thm:main}). 
\begin{conjecture}\label{conj:final}
Let $ \ell \ge 4$. 
Let the random variable~$m$ be the total number of steps in the high-girth triple-process that produces a
partial Steiner system with girth greater than $\ell$. 
Then,  with probability~$1-o(1)$, 
\[ \binom{n}{2} - 3m = n^{3/2+o(1)}. \]
\end{conjecture}
\noindent
In this paper we showed that~$\binom{n}{2} - 3m= O(n^{2-\beta})$, without making any attempt to optimize the constant~$\beta=\beta(\ell)>0$. 
As a first step towards \refConj{conj:final}, there are two natural ways to improve~$\beta$: (a)~to sharpen the hypergraph extension bound~\eqref{eq:thm:ext} by refining the simple union bound based inequality~\eqref{eq:lem:edges}, and (b)~to establish self-correcting estimates for the key variables, as in random triangle removal~\cite{BFL0,BFL}. 
It would also be interesting to understand (the early evolution of) the high-girth triple-process with~$\ell=\ell(n) \to \infty$.

We close by noting that our results on the high-girth triple-process suggest a lower bound
on the number of high-girth Steiner triple systems, following the argument of Keevash on
counting Steiner triple systems~\cite{counting}.  For concreteness, consider the case of $ \ell = 6$;
so we are interested in counting the number of Steiner triple systems on $n$ vertices that contain
no copy of the so-called Pasch configuration~\cite{AntiPasch}. It follows from our results that the number of choices 
at step~$i+1$ of the process, for $i =0 ,\dots, (1-o(1)) n^2/6$, is roughly 
\[ 
|Q(i)| \approx \hq(i/n^2) \approx \frac{n^3}{6} \cdot \left( 1 -\frac{ 6i}{n^2} \right)^3 \cdot \exp \left\{ - \left( \frac{ 6 i}{ n^2} \right)^3 \right\} . \]
Assuming that we can establish sufficient control on the error terms, the 
number of ways to complete the process is then (roughly) at least
\[ N_1 := 
\exp \left\{ \frac{n^2}{6}  \log \left( \frac{n^3}{6} \right)  + 3 \sum_{i=1}^{n^2/6} \log \left(1- \frac{ 6i}{n^2} \right)     - \frac{6^3}{ n^6} \sum_{i=1}^{n^2/6} i^3   \right\}
\approx \exp \left\{  \frac{n^2}{6} \left[  \log\left(\frac{n^3}{6}\right)  -3 - \frac{1}{4}  \right]  \right\} .
\]
On the other hand, a given Steiner triple system can be realized as roughly
\[ N_2 := \left( \frac{n^2}{6} \right)! \approx \exp \left\{ \frac{n^2}{6} \left[ \log\left(\frac{n^2}{6}\right) - 1 \right] \right\} \]
sequences of this kind.
If we assume that each triple system produced by the high-girth triple-process can be completed to a Pasch-free 
Steiner triple system (which would
require not only an affirmative answer of the Erd\H os--Question~\ref{quest:proper}, but also a version for pseudo-random 
triple systems), then this suggests that the number of Pasch-free Steiner triple systems is approximately at least 
\begin{equation}\label{eq:Pasch}
\frac{ N_1}{N_2} \approx \exp \left\{  \frac{ n^2}{6} \left[ \log n - 2 - \frac{1}{4}  \right] \right\}  = \biggpar{\frac{n}{e^{9/4}} }^{ \frac{ n^2}{6}},
\end{equation}
and it would be interesting to know whether their number is indeed~$\bigpar{(1+o(1))n/e^{9/4}}^{n^2/6}$.

\normalsize
\end{document}